\documentclass[conference, 10pt, twocolumn]{ieeeconf}       
\bstctlcite{bstctl:etal, bstctl:nodash, bstctl:simpurl}
\def\BibTeX{{\rm B\kern-.05em{\sc i\kern-.025em b}\kern-.08em
    T\kern-.1667em\lower.7ex\hbox{E}\kern-.125emX}}
    
\usepackage[left=54pt, right=54pt,bottom=54pt, top=54pt]{geometry}

\usepackage{amsmath,mathrsfs,amsfonts,amssymb,graphicx,epsfig}
 \usepackage{amsthm}
\usepackage{subcaption}
\usepackage{color,multirow,rotating}
\usepackage{algorithm,algpseudocode,algorithmicx}
\usepackage{cite,url,framed,bm,balance,nicematrix}
\usepackage{stmaryrd}

\newtheorem{theorem}{Theorem}

\newtheorem{definition}{Definition}

\newtheorem{proposition}{Proposition}
\newtheorem{remark}{Remark}

\usepackage{tikz}
\usetikzlibrary{calc,trees,positioning,arrows,chains,shapes.geometric,decorations.pathreplacing,decorations.pathmorphing,shapes, matrix,shapes.symbols}

\newcommand{\differential}{{\rm{d}}}

\makeatletter
\newcommand{\dotminus}{\mathbin{\text{\@dotminus}}}

\newcommand{\@dotminus}{%
  \ooalign{\hidewidth\raise1ex\hbox{.}\hidewidth\cr$\m@th-$\cr}%
}

\allowdisplaybreaks

\title{\LARGE\textbf{The Ground Cost for Optimal Transport of Angular Velocity}
}

\author{Karthik Elamvazhuthi and Abhishek Halder
\thanks{Karthik Elamvazhuthi is with the Los Alamos National Laboratory, Los Alamos, NM 87544, USA, {\tt\small{karthikevaz@lanl.gov}}.}
\thanks{Abhishek Halder is with the Department of Aerospace Engineering, Iowa State University, Ames, IA 50011, USA, {\tt\small{ahalder@iastate.edu}}.%
}}

\IEEEoverridecommandlockouts
\begin{document}
\bstctlcite{IEEE_b:BSTcontrol}
\maketitle
\thispagestyle{empty}
\pagestyle{empty}

\begin{abstract}
We revisit the optimal transport problem over angular
velocity dynamics given by the controlled Euler equation. The solution
of this problem enables stochastic guidance of spin states of a rigid
body (e.g., spacecraft) over a hard deadline constraint by transferring
a given initial state statistics to a desired terminal state
statistics. This is an instance of generalized optimal transport over
a nonlinear dynamical system. While prior work has reported
existence-uniqueness and numerical solution of this dynamical optimal
transport problem, here we present structural results about the
equivalent Kantorovich a.k.a. optimal coupling formulation. Specifically, we focus on deriving the ground cost for the associated Kantorovich optimal coupling formulation. The ground cost is equal to the cost of transporting unit amount of mass from a specific realization of the initial or source joint probability measure to a realization of the  terminal or target joint probability measure, and determines the Kantorovich formulation. Finding the ground cost leads to solving a structured deterministic nonlinear optimal control problem, which is shown to be amenable to an analysis technique pioneered by Athans et al. We show that such techniques have broader applicability in determining the ground cost (thus Kantorovich formulation) for a class of generalized optimal mass transport problems involving nonlinear dynamics with translated norm-invariant drift.

\end{abstract}


\section{Introduction}\label{sec:Introduction}
We consider the optimal mass transport (OMT) problem \cite{villani2021topics,villani2008optimal} over the rigid body angular velocity dynamics given by the Euler equation:
\begin{subequations}
\begin{align}
J_1 \dot{\omega}_1 &=\left(J_2-J_3\right) \omega_2 \omega_3+\tau_1, \\
J_2 \dot{\omega}_2 &=\left(J_3-J_1\right) \omega_3 \omega_1+\tau_2, \\
J_3 \dot{\omega}_3 &=\left(J_1-J_2\right) \omega_1 \omega_2+\tau_3.
\end{align}
\label{ControlledODEComponentwise}
\end{subequations}
In \eqref{ControlledODEComponentwise}, the vector $$\bm{\omega}:=(\omega_1,\omega_2,\omega_3)^{\top}\in\mathbb{R}^{3}$$ is the angular velocity, the parameter vector
$$\bm{J}:=\left(J_1,J_2,J_3\right)^{\top}\in\mathbb{R}^{3}_{>0}$$ 
comprises of the principal moments of inertia, and $$\bm{\tau}:=(\tau_1,\tau_2,\tau_3)^{\top}\in\mathbb{R}^{3}$$ is the vector of torque inputs along the principal axes. 

In our context, the ``mass" in OMT refers to probability mass, as we are concerned with optimally steering the joint statistics of the initial angular velocity to that of the terminal angular velocity via \eqref{ControlledODEComponentwise} over a given finite horizon.

Specifically, the problem is to find optimal feedback policy for the input torque $\bm{\tau}$ to accomplish \emph{minimum effort} transfer of a given joint probability density function (PDF) of initial angular velocity\footnote{the symbol $\sim$ denotes ``follows the law or statistics"} $\bm{\omega}_0 \sim \rho_{0}$ to another given PDF of terminal or final angular velocity $\bm{\omega}_{\mathrm{f}} \sim \rho_{\mathrm{f}}$ over a given deadline $[0,t_{\mathrm{f}}]$. In other words, the problem data $\rho_{0},\rho_{\mathrm{f}}$ satisfy
$$\rho_{0},\rho_{\mathrm{f}}\geq 0, \quad \int_{\mathbb{R}^{3}}\rho_{0}\left(\bm{\omega}_{0}\right)\differential\bm{\omega}_{0} = 1, \quad \int_{\mathbb{R}^{3}}\rho_{\mathrm{f}}\left(\bm{\omega}_{\mathrm{f}}\right)\differential\bm{\omega}_{\mathrm{f}} = 1.$$

This is a generalized optimal mass transport (GOMT) problem since the classical OMT \cite{benamou2000computational} is over the vector of integrators, i.e., has full control authority. In contrast, our GOMT problem requires optimally steering the state statistics over a controlled nonlinear dynamical constraint \eqref{ControlledODEComponentwise}.

\subsubsection*{Motivation} There exists clear engineering motivation behind the GOMT problem over the angular velocity dynamics given by the Euler equation. Solving this problem amounts to designing a guidance policy to steer the initial stochastic spin or angular velocity state of a rigid body (e.g., spacecraft) to a desired terminal stochastic spin state. In spacecraft spin stabilization, the initial angular velocity state is stochastic in practice because of estimation errors. Thus, the specification of initial PDF is natural. The desired terminal PDF can be seen as an acceptable statistical accuracy specification. For instance, instead of (deterministically) enforcing some constant (e.g., zero) terminal angular velocity, one may specify a tall Gaussian endpoint PDF around that desired constant angular velocity state. The minimum effort objective over the given deadline directly translates to cost, e.g., when the actuation involves thruster firings.    

\subsubsection*{The problem}
For convenience, let us define the state vector
\begin{align}
\bm{x} := \bm{J}\odot\bm{\omega}, \;\text{where}\;\odot\;\text{denotes elementwise product},
\label{defStateVector}    
\end{align}
and the control vector
\begin{align}
\bm{u} := \bm{\tau}.
\label{defControlVector}    
\end{align}
Also define the constants
\begin{equation}
\alpha := \dfrac{1}{J_3} - \dfrac{1}{J_2},\quad
\beta := \dfrac{1}{J_1} - \dfrac{1}{J_3},\quad
\gamma := \dfrac{1}{J_2} - \dfrac{1}{J_1}.
\label{defalphabetagamma}     
\end{equation}
Note that $$\alpha + \beta + \gamma = 0.$$ 
The GOMT problem of interest is the following:
\begin{align}
\underset{\left(\xi,\bm{u}\right)}{\arg\inf} \displaystyle\int_{0}^{t_{\mathrm{f}}}\int_{\mathbb{R}^{3}}\frac{1}{2}\bm{u}^{\top}\bm{u}\;\:\xi(t,\bm{x})\differential\bm{x}\:\differential t\label{GOMTobj}
\end{align}
subject to the controlled dynamical constraint
\begin{subequations}
\begin{align}
\dot{x}_1 &=\alpha x_2 x_3 + u_1,\\
\dot{x}_2 &=\beta x_3 x_1 + u_2,\\
\dot{x}_3 &=\gamma x_1 x_2 + u_3,
\end{align}
\label{xdynamics}
\end{subequations}
and endpoint statistics constraints
\begin{align}
\bm{x}_{0}\sim\xi_{0} := \dfrac{\rho_{0}\left(\bm{x}\oslash \bm{J}\right)}{J_1 J_2 J_3}, \quad \bm{x}_{\mathrm{f}}\sim\xi_{\mathrm{f}} := \dfrac{\rho_{\mathrm{f}}\left(\bm{x}\oslash \bm{J}\right)}{J_1 J_2 J_3},
\label{defxi0xi1}    
\end{align}
where $\oslash$ denotes elementwise division. Here, $\xi$ denotes the joint PDF of $\bm{x}$, the controlled state vector. The new endpoint PDFs $\xi_{0},\xi_{\mathrm{f}}$ are the pushforwards of $\rho_{0},\rho_{\mathrm{f}}$ under \eqref{defStateVector}. Due to the distributional constraints in \eqref{defxi0xi1}, the GOMT is a stochastic optimal control problem even though there is no process noise in \eqref{xdynamics}. 

The $\arg\inf$ in \eqref{GOMTobj} is over
\begin{align}
\left(\xi,\bm{u}\right)\in\mathcal{P}_{0{\mathrm{f}}}\times\mathcal{U},
\label{FeasibleSet}
\end{align}
where 
$\mathcal{P}_{0{\mathrm{f}}}$ is the space of all PDF-valued curves over $\mathbb{R}^{3}$ which are absolutely continuous w.r.t. $t\in[0,t_{\mathrm{f}}]$, and have endpoints $\xi_0$ at $t=0$, and $\xi_{\mathrm{f}}$ at $t=t_{\mathrm{f}}$. The set $\mathcal{U}$ is the collection of all finite energy Markovian control policies, i.e.,
$$\mathcal{U}:=\{\bm{u}:[0,t_{\mathrm{f}}]\times\mathbb{R}^{3}\mapsto\mathbb{R}^{3}\mid \|\bm{u}\|_2^2 < \infty\}.$$
We mention here that in the objective \eqref{GOMTobj}, one may replace the Lagrangian $\frac{1}{2}\bm{u}^{\top}\bm{u}$ by its weighted variant $\frac{1}{2}\bm{u}^{\top}\bm{Ru}$ for fixed positive definite $\bm{R}$. Since the weighted Lagrangian can be reduced to the unweighted version by simply re-scaling the control, we will not elaborate this generalization. 

Assuming the endpoint PDFs $\rho_0,\rho_{\mathrm{f}}$ have finite second moments, the existence-uniqueness for the solution of \eqref{GOMTobj}, \eqref{xdynamics}, \eqref{defxi0xi1} was proved in \cite[Sec. III-B]{yan2023optimal} using Figalli's theory \cite{figalli2007optimal} of OMT with Tonelli Lagrangian \cite[p. 118]{villani2008optimal}. So the $\arg\inf$ in \eqref{GOMTobj} can be replaced by $\arg\min$. Its numerical solution was found in \cite[Sec. V-VII]{yan2023optimal} using stochastic regularization followed by training a neural network informed by a system of boundary-coupled PDEs comprising of the necessary conditions for optimality.

In this work, we focus on the \emph{static} GOMT a.k.a. the Kantorovich formulation \cite{kantorovich1942translocation} that is equivalent to our dynamic GOMT problem \eqref{GOMTobj}, \eqref{xdynamics}, \eqref{defxi0xi1}. The static formulation is of the form
\begin{align}
\underset{\pi \in \Pi_2\left(\xi_0, \xi_{\mathrm{f}}\right)}{\arg \inf } \int_{\mathbb{R}^3 \times \mathbb{R}^3} c(\boldsymbol{x}_{0}, \boldsymbol{x}_{\mathrm{f}}) \mathrm{d} \pi(\boldsymbol{x}_{0}, \boldsymbol{x}_{\mathrm{f}}) 
\label{staticOMT}
\end{align}
for suitable \emph{ground cost} 
\begin{align}
c:\mathbb{R}^{3}\times\mathbb{R}^{3}\mapsto\mathbb{R}_{\geq 0},
\label{GroundCostMapping}    
\end{align}
and the feasible set
\begin{align}
&\Pi_{2}\left(\xi_0, \xi_{\mathrm{f}}\right):=\{\text{joint probability measure} \;\pi(\boldsymbol{x}_{0}, \boldsymbol{x}_{\mathrm{f}})\;\text{with}\nonumber\\
&\text{finite second moment}\;\mid \bm{x}_{0}\sim\xi_0,\;\bm{x}_{\mathrm{f}}\sim\xi_{\mathrm{f}}\}.    
\label{DefPi2}
\end{align}
In other words, the minimization in \eqref{staticOMT} is over all possible couplings of the endpoint marginals $\xi_0,\xi_{\mathrm{f}}$.

What is particularly appealing about formulation \eqref{staticOMT} is that it is a linear program for some to-be-determined ground cost $c$. The ground cost quantifies the cost of steering unit mass from a fixed angular velocity $\bm{x}_{0}$ to another fixed angular velocity $\bm{x}_{\mathrm{f}}$. Furthermore, this ground cost $c$ can be found as the optimal value of the associated deterministic optimal control problem: 
\begin{equation}
\begin{gathered}
c(\bm{x}_{0},\bm{x}_{\mathrm{f}}) = \underset{\bm{u}}{\min}\int_{0}^{t_{\mathrm{f}}} \frac{1}{2}\bm{u}^{\top}\bm{u}\:\differential t\\
\text{subject to} \quad \eqref{xdynamics},\quad \bm{x}(t=0)=\bm{x}_{0},\quad \bm{x}(t=t_{\mathrm{f}})=\bm{x}_{\mathrm{f}}.
\end{gathered}
\label{OCPforGroundCost}
\end{equation}
The equivalence between problem \eqref{staticOMT} for the cost \eqref{OCPforGroundCost} and \eqref{GOMTobj} follows from the equivalence between static and dynamic formulations of optimal transport with optimal control costs (see e.g., \cite[Thm. 3.10]{elamvazhuthi2024benamou}).
\begin{remark}
The $c$ for classical OMT can be recovered from \eqref{OCPforGroundCost} by replacing \eqref{xdynamics} with the simpler $$\dot{\bm{x}}=\bm{u}.$$ Then, an easy computation via Pontryagin's minimum principle yields 
\begin{align}
c_{\mathrm{classical}}(\bm{x}_0,\bm{x}_{\mathrm{f}}) = \dfrac{\|\bm{x}_{0} - \bm{x}_{\mathrm{f}}\|_2^2}{2\:t_{\mathrm{f}}},
\label{cclassical}
\end{align}
the scaled squared Euclidean distance. However, such direct computation becomes unwieldy for \eqref{xdynamics}.
\end{remark}

\subsubsection*{Contributions} We make the following concrete contributions.

\begin{itemize}
\item Inspired by certain analysis techniques pioneered by Athans et al. \cite{athans1963time}, we derive results for the optimal value of the deterministic optimal control problem \eqref{OCPforGroundCost}, which serves as the ground cost $c$ for the Kantorovich GOMT formulation \eqref{staticOMT}. These analysis techniques combine different variants of the Cauchy-Schwarz inequality, and bypass Pontryagin's minimum principle. The latter is difficult to analyze for this problem due to dynamical nonlinearities. 

\vspace*{0.1in}

\item We show that the aforesaid techniques have broader applicability--beyond the Eulerian angular velocity dynamics--in determining the ground cost for a class of GOMT problems over fully-actuated nonlinear systems with translated norm-invariant drift. We prove that the ground cost for such GOMT problems coincide with the classical OMT ground cost, i.e., the scaled Euclidean distance with the important difference that the geodesics are no longer straight lines. 
\end{itemize}

\subsubsection*{Organization} In Sec. \ref{sec:MainResults}, we study the Eulerian ground cost which is the minimum value for problem \eqref{OCPforGroundCost}. Specifically, Sec. \ref{subsec:boundwellposed} provides an upper bound for this ground cost, and uses this bound to establish that the corresponding static GOMT \eqref{staticOMT} is well-posed.
A particular feasible controller is constructed in Sec. \ref{subsec:SpecificFeasibleController}, and a cost inequality for the same is derived in Sec. \ref{subsec:optimalityEuler}. 

In Sec. \ref{sec:GroundCostNormInv}, we specialize the results for the particular controller from Sec. \ref{sec:MainResults} to exactly deduce the ground costs for a class of GOMT problems when the prior drift satisfies certain translated norm-invariance property. For this class of systems, the feasible controller from Sec. \ref{subsec:SpecificFeasibleController} takes a form that is shown to saturate the cost inequality derived in Sec. \ref{subsec:optimalityEuler}, thus guaranteeing optimality.

Sec. \ref{sec:Conclusions} concludes the paper.

\section{Ground cost for Optimal Transport Over Eulerian Angular Velocity Dynamics}\label{sec:MainResults}

For $\bm{x}$ as in \eqref{defStateVector}, consider a change of co-ordinate:
\begin{align}
\bm{z} := \bm{x} - \bm{x}_{\mathrm{f}}.
\label{xtoz}    
\end{align}
Steering $\bm{x}_0$ to $\bm{x}_{\mathrm{f}}$ over $[0,t_{\mathrm{f}}]$ is then equivalent to steering $\bm{z}_0 := \bm{x}_0 - \bm{x}_{\mathrm{f}}$ to $\bm{z}_{\mathrm{f}} := \bm{0}$ over $[0,t_{\mathrm{f}}]$. 

In the new co-ordinate, the dynamics \eqref{xdynamics} gets mapped to
\begin{align}
\begin{pmatrix}
\dot{z}_1\\ 
\dot{z}_2\\
\dot{z}_3
\end{pmatrix} &= \underbrace{\begin{pmatrix}
\alpha z_2 z_3\\
\beta z_3 z_1\\
\gamma z_1 z_2
\end{pmatrix}}_{\bm{f}_0(\bm{z})} + \underbrace{\begin{bmatrix}
0 & \alpha x_{\mathrm{f}3} & \alpha x_{\mathrm{f}2}\\
\beta x_{\mathrm{f}3} & 0 & \beta x_{\mathrm{f}1}\\
\gamma x_{\mathrm{f}2} & \gamma x_{\mathrm{f}1} & 0
\end{bmatrix}\begin{pmatrix}
z_1\\ 
z_2\\
z_3
\end{pmatrix}}_{\bm{Az}}\nonumber\\
&\qquad\qquad\qquad\qquad+\underbrace{\begin{pmatrix}
\alpha x_{\mathrm{f}2}x_{\mathrm{f}3}\\ 
\beta x_{\mathrm{f}3}x_{\mathrm{f}1}\\
\gamma x_{\mathrm{f}1}x_{\mathrm{f}2}
\end{pmatrix}}_{\bm{b}} + \begin{pmatrix}
u_1\\
u_2\\
u_3
\end{pmatrix},
\label{zdynamics}
\end{align}
which differs from \eqref{xdynamics} by an additional affine drift $\bm{Az}+\bm{b}$, wherein $\left(\bm{A},\bm{b}\right)$ is a constant matrix-vector pair as above. We emphasize that the pair $\left(\bm{A},\bm{b}\right)$ is solely parameterized by the (realization of the) terminal state vector $\bm{x}_{\mathrm{f}}$.

We write \eqref{zdynamics} succinctly as
\begin{align}
\dot{\bm{z}}=\bm{f}_{0}(\bm{z}) + \bm{Az} + \bm{b} + \bm{u}.
\label{zdynamicsVectorForm}    
\end{align}
\begin{remark}
We clarify that the change-of-coordinate \eqref{xtoz} only pertains to finding the ground cost $c$ by solving the associated deterministic optimal control problem \eqref{OCPforGroundCost}. In particular, it does not induce a pushforward of probability measure. The optimal cost of problem \eqref{OCPforGroundCost} helps construct the static GOMT problem \eqref{staticOMT}. 
\end{remark}
\begin{remark}\label{remark:Abzero}
From \eqref{zdynamics}, note that if $\bm{x}_{\mathrm{f}}=\bm{0}$, then $\bm{A}=\bm{0}$, $\bm{b}=\bm{0}$.
\end{remark}

\subsection{Bound and Well-posedness}\label{subsec:boundwellposed}
We next derive a bound (Proposition \ref{prop:BoundEuler}) for the Eulerian ground cost $c$, i.e., the optimal value for problem \eqref{OCPforGroundCost}. The idea behind this bound is to make two-fold use of the observation in Remark \ref{remark:Abzero}. This allows us to bound the ground cost for arbitrary endpoint, using the ground costs for two associated problems, both with zero endpoints--one in forward time and another in reverse time. 

We then use this bound to prove well-posedness (Theorem \ref{Thm:wellposesStaticGOMT}) for the static GOMT problem \eqref{staticOMT}. 

\begin{proposition}\label{prop:BoundEuler}
The ground cost $c(\bm{x}_0,\bm{x} _{\mathrm{f}})$ in \eqref{OCPforGroundCost} satisfies 
\[c(\bm{x}_0,\bm{x} _{\mathrm{f}}) \leq\frac{\|\bm{x}_0\|_{2}^2}{t _{\mathrm{f}}} + \frac{\|\bm{x} _{\mathrm{f}}\|_{2}^2}{t _{\mathrm{f}}}\]

for all $\bm{x}_0,\bm{x} _{\mathrm{f}} \in \mathbb{R}^3$.
\label{cstnd}
\end{proposition}

\begin{proof}
We will divide the time horizon $[0,t _{\mathrm{f}}]$ into two halves
$[0,t _{\mathrm{f}}/2]$ and $(t _{\mathrm{f}}/2,t _{\mathrm{f}}]$, and construct a control over each.

For $\bm{x}_{\mathrm{f}}=\bm{0}$, we have $\bm{A}=\bm{0}$, $\bm{b}=\bm{0}$ (see Remark \ref{remark:Abzero}), and the drift $\bm{f}_{0}$ in \eqref{zdynamics} is norm invariant because $\langle\bm{f}_{0}(\bm{x}),\bm{x}\rangle = 0$, see \cite{athans1963time}. Then, \cite{athans1963time} constructively shows that there exists optimal control $\bm{u}_1: [0,\frac{t _{\mathrm{f}}}{2}] \rightarrow \mathbb{R}^3$ such that for the system \eqref{xdynamics}, the controlled state satisfies $\bm{x}(\frac{t _{\mathrm{f}}}{2}) = 0 $ and $\int_0^{\frac{t _{\mathrm{f}}}{2}} \frac{1}{2} \bm{u}_1^{\top}\bm{u}_1 \:\differential t = \frac{\|\bm{x}_0\|_{2}^2}{t _{\mathrm{f}}}$.

On the other hand, for $\bm{x}_{\mathrm{f}}=\bm{0}$, the time reversal of the system \eqref{xdynamics} is expressed as
\begin{equation}
\quad\dot{\bm{x}}_{r}=-\bm{f}_{0}(\bm{x}_{r})-\bm{u}_{r},
\label{timereversedODE}
\end{equation}
where $\bm{x}_{r},\bm{u}_{r}$ denote the state and control for the time-reversed system. Notice that \eqref{timereversedODE} is also norm invariant since $\langle-\bm{f}_{0}(\bm{x}_{r}),\bm{x}_{r}\rangle = 0$. Therefore, one can find a control $\bm{u}_r:[0,t _{\mathrm{f}}] \rightarrow \mathbb{R}^3$ such that 
$\bm{x}_{r}(\frac{t _{\mathrm{f}}}{2}) = 0 $ and $\int_0^{\frac{t_{\mathrm{f}}}{2}}  \frac{1}{2} \bm{u}_r^{\top}\bm{u}_r \:\differential t = \frac{\|\bm{x} _{\mathrm{f}}\|_{2}^2}{t _{\mathrm{f}}}$. Letting $$\bm{u}_2(t):=\bm{u}_r(t _{\mathrm{f}} - t)\quad\forall t \in [0,t _{\mathrm{f}}/2],$$ it then follows that for the system 
\begin{equation}
\quad\dot{\bm{x}}=\bm{f}_{0}(\bm{x})-\bm{u}_2, 
\end{equation}
we have $\bm{x}(0) = \bm{0}$ and $\bm{x}(\frac{t _{\mathrm{f}}}{2}) = \bm{x} _{\mathrm{f}}$.

Now, consider the control $\bm{u}:[0,t _{\mathrm{f}}] \rightarrow \mathbb{R}^3$ given by
\begin{equation}
\bm{u}(t) := 
\begin{cases}
\bm{u_1}(t),~~~t \in [0,\frac{t _{\mathrm{f}}}{2}],\\
\bm{u_2}(t),~~~t \in (\frac{t _{\mathrm{f}}}{2},t _{\mathrm{f}}].
\end{cases}
\end{equation}
For this choice of control, we clearly have that 
$\bm{x}(0) = \bm{x}_0$, $\bm{x}(t _{\mathrm{f}}) = \bm{x} _{\mathrm{f}}$, and by construction 
\[\int_{0}^{t_{\mathrm{f}}} \frac{1}{2}\bm{u}^{\top}\bm{u}\:\differential t =\frac{\|\bm{x}_0\|_{2}^2}{t _{\mathrm{f}}} + \frac{\|\bm{x} _{\mathrm{f}}\|_{2}^2}{t _{\mathrm{f}}}.\]
This concludes the proof.
\end{proof}

\begin{theorem}\label{Thm:wellposesStaticGOMT}
Suppose the endpoint joint PDFs $\xi_{\mathrm{0}}, \xi_{\mathrm{f}}$ have finite second order moments. Then, the static GOMT problem \eqref{staticOMT} is well-posed.
\end{theorem}
\begin{proof}
Under the stated assumption, a feasible solution trivially exists by considering the product distribution $\pi := \xi_{\mathrm{0}} \otimes \xi_{\mathrm{f}}$. It follows from the bounds derived in Proposition  \ref{cstnd} that $\int c(\bm{x}_0,\bm{x}_{\mathrm{f}}) \pi(\bm{x}_0,\bm{x}_{\mathrm{f}}) <\infty$. The set of measures with bounded second order moments is compact. Hence, existence of solution $\pi \mapsto \int c(\bm{x}_0,\bm{x}_{\mathrm{f}}) \pi(\bm{x}_0,\bm{x}_{\mathrm{f}})$ for \eqref{staticOMT} follows from continuity of the functional over a compact set.
\end{proof}

\subsection{A Feasible Control $\bm{u}^{*}$}\label{subsec:SpecificFeasibleController}
We next construct a specific feasible control for problem \eqref{OCPforGroundCost}, denoted as $\bm{u}^{*}$, which will turn out to be instructive for the development that follows.

\begin{figure*}[th]
\centering
\includegraphics[width=\linewidth]{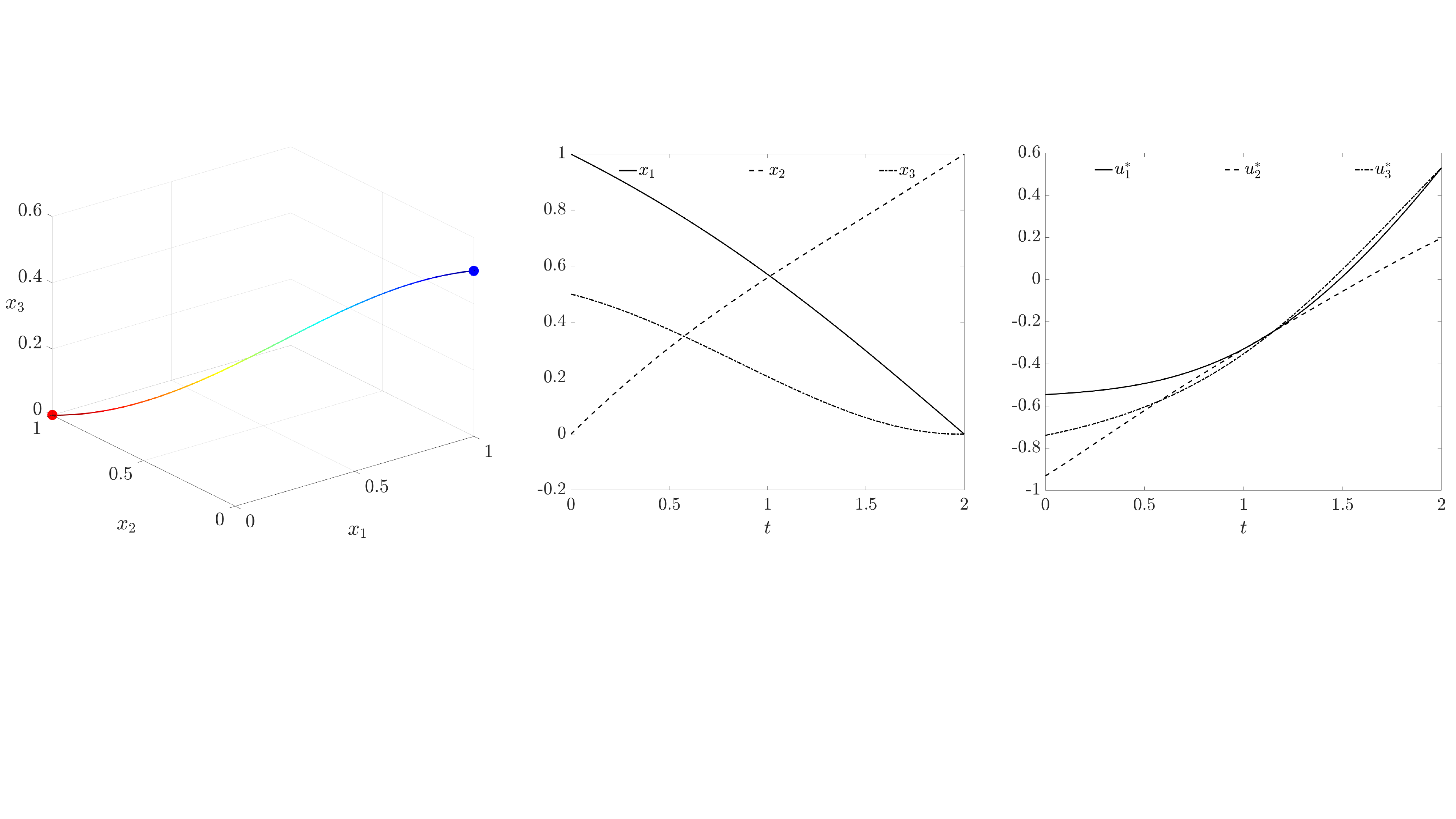}
\caption{\emph{Left:} The steering of \textcolor{blue}{$\bm{x}_0=(1,0,0.5)^{\top}$} to \textcolor{red}{$\bm{x}_{\mathrm{f}}=(0,1,0)^{\top}$} subject to \eqref{xdynamics} with $J_1 = 1, J_2 = 2, J_3=3$, over $[0,t_{\mathrm{f}}]=[0,2]$ using the feasible control $\bm{u}^{*}\left(\bm{x}-\bm{x}_{\mathrm{f}}\right)$ where $\bm{u}^{*}$ is given by \eqref{GeneralizedAthansFeasibleControl}; \emph{center:} time evolution of the corresponding state components; \emph{right:} time evolution of the corresponding control components.}
\label{fig:feasible}
\end{figure*}

\begin{proposition}
The control 
\begin{align}
\bm{u}^{*}(\bm{z}) := \left(-\dfrac{\|\bm{z}_{0}\|_2}{t_{\mathrm{f}}} - \dfrac{\langle\bm{z},\bm{Az}+\bm{b}\rangle}{\|\bm{z}\|_2}\right)\dfrac{\bm{z}}{\|\bm{z}\|_2}
\label{GeneralizedAthansFeasibleControl}
\end{align}
is guaranteed to steer the nonlinear dynamics \eqref{zdynamics} from arbitrary $\bm{z}_{0}$ to $\bm{z}_{\mathrm{f}} := \bm{0}$ over $[0,t_{\mathrm{f}}]$.  
\end{proposition}
\begin{proof}
The idea is to enforce
\begin{align}
\dfrac{\differential}{\differential t}\|\bm{z}\|_2 = k \quad\forall t\in[0,t_{\mathrm{f}}]
\label{NormDerivativeConstant}  
\end{align}
for some constant $k$, since then
\begin{align}
\|\bm{z}\|_2 = kt + \|\bm{z}_0\|_2 \quad\forall t\in[0,t_{\mathrm{f}}].
\label{NormEquality}    
\end{align}
Evaluating \eqref{NormEquality} at $t=t_{\mathrm{f}}$ determines the constant
\begin{align}
k = -\dfrac{\|\bm{z}_{0}\|_2}{t_{\mathrm{f}}},
\label{kintermsofzf}    
\end{align}
and we have
\begin{align}
\|\bm{z}\|_2 = \left(1 - \dfrac{t}{t_{\mathrm{f}}}\right) \|\bm{z}_0\|_2 \quad\forall t\in[0,t_{\mathrm{f}}].
\label{NormOfzUnderCandidateControl}    
\end{align}
Since the state norm is zero if and only if the state is zero, so a control accomplishing \eqref{NormDerivativeConstant}, if found, satisfies the endpoint constraints. 

Notice that
\begin{align}
\dfrac{\differential}{\differential t}\|\bm{z}\|_2 &= \dfrac{\differential}{\differential t}\sqrt{\langle\bm{z},\bm{z}\rangle}\nonumber\\
&= \dfrac{\langle\dot{\bm{z}},\bm{z}\rangle}{\|\bm{z}\|_2}\nonumber\\
&=\dfrac{\langle \bm{Az}+\bm{b},\bm{z}\rangle + \langle\bm{u},\bm{z}\rangle}{\|\bm{z}\|_2},
\label{DerivativeOfNorm}
\end{align}
where the last equality uses \eqref{zdynamicsVectorForm} and that $\langle \bm{f}_{0}(\bm{z}),\bm{z}\rangle = 0$. Using \eqref{DerivativeOfNorm}, we now show that a control of the form 
\begin{align}
\bm{u} = g(\bm{z})\dfrac{\bm{z}}{\|\bm{z}\|_2}
\label{FeasibleControlGenericForm}    
\end{align}
indeed accomplishes \eqref{NormDerivativeConstant} for some to-be-determined scalar function $g(\bm{z})$.

Substituting \eqref{FeasibleControlGenericForm} in \eqref{DerivativeOfNorm}, and invoking \eqref{NormDerivativeConstant}, we find
\begin{align}
g(\bm{z}) = k - \dfrac{\langle\bm{z},\bm{Az}+\bm{b}\rangle}{\|\bm{z}\|_2}.
\label{feasiblecontrolform}    
\end{align}
Substituting for $k$ in \eqref{feasiblecontrolform} from \eqref{kintermsofzf}, yields $\bm{u}=\bm{u}^{*}$ given by \eqref{GeneralizedAthansFeasibleControl}.
\end{proof}
\begin{remark}
By construction, if $\bm{u}^{*}$ is given by \eqref{GeneralizedAthansFeasibleControl}, then $\bm{u}^{*}\left(\bm{x}-\bm{x}_{\mathrm{f}}\right)$ is guaranteed to steer $\bm{x}_0$ to $\bm{x}_{\mathrm{f}}$ subject to \eqref{xdynamics} over $[0,t_{\mathrm{f}}]$. See Fig. \ref{fig:feasible}.   
\end{remark}
Except for special cases (Remark \ref{RemarkWhenustaroptimal}), the feasible $\bm{u}^{*}$ is suboptimal. We next deduce an inequality relating the cost associated with an arbitrary feasible control $\widetilde{\bm{u}}$ for problem \eqref{OCPforGroundCost}. This inequality will be crucial in Sec. \ref{sec:GroundCostNormInv}. Our analysis technique is inspired by the work of Athans et al. \cite{athans1963time}, and extends the same.

\subsection{Cost Inequality}\label{subsec:optimalityEuler}
For fixed $\bm{x}_{0},\bm{x}_{\mathrm{f}}$ (thus fixed $\bm{z}_{0}$), consider a feasible controller $\widetilde{\bm{u}}$ different from $\bm{u}^{*}$ in \eqref{GeneralizedAthansFeasibleControl}. 
Let the associated (non-identical) state trajectories be $\widetilde{\bm{z}}(t)$ and $\bm{z}^{*}(t)$, respectively. Since both controllers are feasible, $\widetilde{\bm{z}}(0) = \bm{z}^{*}(0)=\bm{z}_0$ and $\widetilde{\bm{z}}(t_{\mathrm{f}}) = \bm{z}^{*}(t_{\mathrm{f}})=\bm{0}$.

Notice that $$\dfrac{\differential}{\differential t}\|\widetilde{\bm{z}}\|_2 = \dfrac{\langle\dot{\widetilde{\bm{z}}},\widetilde{\bm{z}}\rangle}{\|\widetilde{\bm{z}}\|_2}=\dfrac{\langle\bm{A}\widetilde{\bm{z}}+\bm{b},\widetilde{\bm{z}}\rangle}{\|\widetilde{\bm{z}}\|_2} + \dfrac{\langle\bm{u},\widetilde{\bm{z}}\rangle}{\|\widetilde{\bm{z}}\|_2}.$$
Applying triangle inequality to above, we obtain
\begin{align}
\bigg\vert \dfrac{\differential}{\differential t}\|\widetilde{\bm{z}}\|_2 \bigg\vert &\leq \dfrac{\big\vert\widetilde{\bm{z}}^{\top}\bm{A}\widetilde{\bm{z}}\big\vert}{\|\widetilde{\bm{z}}\|_2} + \dfrac{\big\vert\langle\bm{b},\widetilde{\bm{z}}\rangle\big\vert}{\|\widetilde{\bm{z}}\|_2} + \dfrac{\big\vert\langle\widetilde{\bm{u}},\widetilde{\bm{z}}\rangle\big\vert}{\|\widetilde{\bm{z}}\|_2}\nonumber\\
&\leq \|\bm{A}\widetilde{\bm{z}}\|_2 + \|\bm{b}\|_2 + \|\widetilde{\bm{u}}\|_2,
\label{TraingleInequality}    
\end{align}
where the last line used the Cauchy-Schwarz inequality. From \eqref{TraingleInequality}, we have
$$\dfrac{\differential}{\differential t}\|\widetilde{\bm{z}}\|_2 \geq - \|\bm{A}\widetilde{\bm{z}}\|_2 - \|\bm{b}\|_2 - \|\widetilde{\bm{u}}\|_2,$$
or equivalently,
\begin{align}
&\|\bm{z}_{0}\|_2 \leq t_{\mathrm{f}}\|\bm{b}\|_2 + \int_{0}^{t_{\mathrm{f}}} \left(\|\bm{A}\widetilde{\bm{z}}\|_2 + \|\widetilde{\bm{u}}\|_2\right)\differential t\nonumber\\
\Rightarrow& \|\bm{z}_{0}\|_2 - t_{\mathrm{f}}\|\bm{b}\|_2 - \int_{0}^{t_{\mathrm{f}}}\|\bm{A}\widetilde{\bm{z}}\|_2\differential t \leq \int_{0}^{t_{\mathrm{f}}}\|\widetilde{\bm{u}}\|_2\differential t.
\label{IntegralOfControlNormLowerBound}
\end{align}

Next, recall the integral version of the Cauchy-Schwarz inequality: for $p,q$ square integrable,
\begin{align}
\left(\int_{0}^{t_{\mathrm{f}}}p(t)q(t)\differential t\right)^{2} \leq \left(\int_{0}^{t_{\mathrm{f}}}p^{2}(t)\differential t\right)\left(\int_{0}^{t_{\mathrm{f}}}q^{2}(t)\differential t\right).
\label{CSintegralversion}    
\end{align}
Specializing \eqref{CSintegralversion} for $p(t)=\|\widetilde{\bm{u}}\|_2$, $q(t)=1$, we get
\begin{align}
\dfrac{1}{2t_{\mathrm{f}}}\left(\int_{0}^{t_{\mathrm{f}}}\|\widetilde{\bm{u}}\|_2\:\differential t\right)^{2} \leq \int_{0}^{t_{\mathrm{f}}} \frac{1}{2}\widetilde{\bm{u}}^{\top}\widetilde{\bm{u}}\:\differential t.
\label{NormIntegralIneq}
\end{align}
Combining \eqref{IntegralOfControlNormLowerBound} and \eqref{NormIntegralIneq} gives the cost inequality
\begin{align}
\dfrac{1}{2t_{\mathrm{f}}}\left(\|\bm{z}_{0}\|_2 - t_{\mathrm{f}}\|\bm{b}\|_2 - \int_{0}^{t_{\mathrm{f}}}\|\bm{A}\widetilde{\bm{z}}\|_2\differential t\right)^{2}\leq \int_{0}^{t_{\mathrm{f}}} \frac{1}{2}\widetilde{\bm{u}}^{\top}\widetilde{\bm{u}}\:\differential t,
\label{LowerBound}
\end{align}
assuming the lower bound in \eqref{IntegralOfControlNormLowerBound} is nonnegative.

\begin{remark}\label{RemarkWhenustaroptimal}
It is particularly interesting to compare the lower bound in \eqref{LowerBound} with $\int_{0}^{t_{\mathrm{f}}} \frac{1}{2}\left(\bm{u}^{*}\right)^{\top}\bm{u}^{*}\:\differential t$. Using \eqref{GeneralizedAthansFeasibleControl}, we have
\begin{align}
\!\int_{0}^{t_{\mathrm{f}}}\!\frac{1}{2}\left(\bm{u}^{*}\right)^{\top}\!\bm{u}^{*}\differential t= \!\!\int_{0}^{t_{\mathrm{f}}}\!\frac{1}{2}\!\left(-\dfrac{\|\bm{z}_{0}\|_2}{t_{\mathrm{f}}} - \dfrac{\langle\bm{z}^{*},\bm{A}\bm{z}^{*}+\bm{b}\rangle}{\|\bm{z}^{*}\|_2}\right)^{2}\!\!\differential t.
\label{CandidateCost}    
\end{align}
In the next Section, we will utilize that for $\bm{A}=\bm{0}$, $\bm{b}=\bm{0}$, the lower bound in \eqref{LowerBound} equals to \eqref{CandidateCost}, i.e., then our feasible controller is also optimal.
\end{remark}


\section{Ground Cost for Optimal Transport Over A Class of Norm-invariant Systems}\label{sec:GroundCostNormInv}
In this Section, we consider GOMT problems over a special class of nonlinear systems. It turns out that the ground cost $c$ for such systems can be exactly determined using the same principles used to construct the feasible controller $\bm{u}^{*}$ in Sec. \ref{subsec:SpecificFeasibleController}.

Specifically, the GOMT problem of our interest is in the form
\begin{subequations}
\begin{align}
&\underset{\left(\xi,\bm{u}\right)}{\arg\inf} \displaystyle\int_{0}^{t_{\mathrm{f}}}\int_{\mathbb{R}^{3}}\frac{1}{2}\bm{u}^{\top}\bm{u}\;\:\xi(t,\bm{x})\differential\bm{x}\:\differential t\label{GOMTnorminvariantobj}\\
&\text{subject to}\quad\dfrac{\partial\xi}{\partial t} + \nabla_{\bm{x}}\cdot\left(\xi\left(\bm{f}(\bm{x})+\bm{u}\right)\right) = 0,\label{GOMTnorminvariantPDE}\\
&\qquad\qquad\quad\xi(t=0,\cdot) = \xi_0, \quad \xi(t=t_{\mathrm{f}},\cdot) = \xi_{\mathrm{f}},\label{GOMTnorminvariantBC}
\end{align}
\label{GOMTnorminvariant}    
\end{subequations}
\!\!where $\nabla_{\bm{x}}\cdot$ denotes divergence w.r.t. $\bm{x}\in\mathbb{R}^{d}$, and $\bm{f}$ has a specific structure. Before delving into that structure, note that the ground cost $c$ in the Kantorovich problem \eqref{staticOMT} associated with \eqref{GOMTnorminvariant} is the optimal value
\begin{subequations}
\begin{align}
&c(\bm{x}_{0},\bm{x}_{\mathrm{f}}) = \underset{\bm{u}}{\min}\int_{0}^{t_{\mathrm{f}}} \frac{1}{2}\bm{u}^{\top}\bm{u}\:\differential t\label{OCPforNormInvariantGroundCostobjx}\\
&\text{subject to} \quad\dot{\bm{x}}=\bm{f}(\bm{x})+\bm{u}, \label{OCPforNormInvariantGroundCostODEx}\\
&\qquad\qquad\quad\bm{x}(t=0)=\bm{x}_{0},\quad\bm{x}(t=t_{\mathrm{f}})=\bm{x}_{\mathrm{f}}.\label{OCPforNormInvariantGroundCostBCx}
\end{align}
\label{OCPforNormInvariantGroundCostx}
\end{subequations}
In the transformed state co-ordinate \eqref{xtoz}, we transcribe \eqref{OCPforNormInvariantGroundCostx} to
\begin{subequations}
\begin{align}
&c(\bm{x}_{0},\bm{x}_{\mathrm{f}}) =c(\bm{z}_0,\bm{0}) = \underset{\bm{u}}{\min}\int_{0}^{t_{\mathrm{f}}} \frac{1}{2}\bm{u}^{\top}\bm{u}\:\differential t\label{OCPforNormInvariantGroundCostobjz}\\
&\text{subject to} \quad\dot{\bm{z}}=\bm{f}(\bm{z}+\bm{x}_{\mathrm{f}})+\bm{u}, \label{OCPforNormInvariantGroundCostODEz}\\
&\qquad\qquad\quad\bm{z}(t=0)=\bm{z}_{0},\quad\bm{z}(t=t_{\mathrm{f}})=\bm{0}.\label{OCPforNormInvariantGroundCostBCz}
\end{align}
\label{OCPforNormInvariantGroundCostz}
\end{subequations}
The definition next imposes additional structure on $\bm{f}$.
\begin{definition}\label{def:translatednorminvariantsystems}[Translated norm-invariant system]
If the drift $\bm{f}$ satisfies
\begin{align}
\langle\bm{f}(\bm{z}+\bm{x}_{\mathrm{f}}),\bm{z}\rangle = 0 \quad\forall\bm{x}_{\mathrm{f}}\in\mathbb{R}^{d},
\label{ZeroInnerProduct}
\end{align}
then we say that \eqref{OCPforNormInvariantGroundCostODEx} satisfies \emph{translated norm-invariance} in the sense the unforced dynamics for \eqref{OCPforNormInvariantGroundCostODEz} in norm-invariant.
\end{definition}

Definition \ref{def:translatednorminvariantsystems} has the following implication. For dynamics \eqref{OCPforNormInvariantGroundCostODEz} where \eqref{ZeroInnerProduct} holds, we have
\begin{align}
\dfrac{\differential}{\differential t}\|\bm{z}\|_2 = \dfrac{\differential}{\differential t}\sqrt{\langle\bm{z},\bm{z}\rangle}=\dfrac{\langle\dot{\bm{z}},\bm{z}\rangle}{\|\bm{z}\|_2}=\dfrac{\langle\bm{u},\bm{z}\rangle}{\|\bm{z}\|_2}.
\label{DerivativeOfNormGeneric}
\end{align}
In particular, \eqref{DerivativeOfNormGeneric} implies that a control $\bm{u}=k\bm{z}/\|\bm{z}\|_2$ for some constant $k$, would result in $\dfrac{\differential}{\differential t}\|\bm{z}\|_2 = k$. Choosing $k$ as per \eqref{kintermsofzf} would then satisfy the endpoint constraints \eqref{OCPforNormInvariantGroundCostz}. In other words, the control  
\begin{align}
\bm{u}^{**}:=-\dfrac{\|\bm{z}_{0}\|_2}{t_{\mathrm{f}}}\dfrac{\bm{z}}{\|\bm{z}\|_2}
\label{ControlProptoUnitVector}
\end{align}
is \emph{feasible} for problem \eqref{OCPforNormInvariantGroundCostz}. 

Interestingly, the control $\bm{u}^{**}$ given by \eqref{ControlProptoUnitVector} is in fact \emph{optimal} for problem \eqref{OCPforNormInvariantGroundCostz}, as demonstrated in the following Proposition \ref{prop:optimalityfortranslatednorminvariantcontrol}. As a consequence, the ground cost for translated norm-invariant systems, denoted as $c_{\mathrm{norm-inv}}$, can be determined in closed form. The proof of Proposition \ref{prop:optimalityfortranslatednorminvariantcontrol} makes use of the inequalities derived in Sec. \ref{subsec:optimalityEuler}.

\begin{proposition}\label{prop:optimalityfortranslatednorminvariantcontrol}
The feasible control \eqref{ControlProptoUnitVector} is the unique minimizer for problem \eqref{OCPforNormInvariantGroundCostz} provided \eqref{OCPforNormInvariantGroundCostODEx} is translated norm-invariant. Consequently,
\begin{align*}
c_{\mathrm{norm-inv}}(\bm{x}_{0},\bm{x}_{\mathrm{f}}) &=c_{\mathrm{norm-inv}}(\bm{z}_0,\bm{0})\nonumber\\
&=\dfrac{\|\bm{z}_{0}\|_2^2}{2\:t_{\mathrm{f}}} = \dfrac{\|\bm{x}_{0} - \bm{x}_{\mathrm{f}}\|_2^2}{2\:t_{\mathrm{f}}}.
\end{align*}
\end{proposition}
\begin{proof}
Notice that the feasible control $\bm{u}^{**}$ in \eqref{ControlProptoUnitVector} is a special case of the $\bm{u}^{*}$ in  \eqref{GeneralizedAthansFeasibleControl} obtained by setting $\bm{A}=\bm{0},\bm{b}=\bm{0}$. 

Now consider an \emph{arbitrary} feasible control $\widetilde{\bm{u}}$ for problem \eqref{OCPforNormInvariantGroundCostz}. Setting $\bm{A}=\bm{0},\bm{b}=\bm{0}$ in the cost inequality \eqref{LowerBound}, we have
$$\dfrac{\|\bm{z}_{0}\|_2^2}{2\:t_{\mathrm{f}}}=\int_{0}^{t_{\mathrm{f}}} \frac{1}{2}\left(\bm{u}^{**}\right)^{\top}\bm{u}^{**}\:\differential t \leq \int_{0}^{t_{\mathrm{f}}} \frac{1}{2}\widetilde{\bm{u}}^{\top}\widetilde{\bm{u}}\:\differential t.$$
Therefore, the feasible control \eqref{ControlProptoUnitVector} is, in fact, a minimizer for problem \eqref{OCPforNormInvariantGroundCostz}. Uniqueness of the minimizer follows from strict convexity of the Lagrangian in \eqref{OCPforNormInvariantGroundCostobjz}. 

Recalling that $\bm{z}_{0}:=\bm{x}_{0} - \bm{x}_{\mathrm{f}}$, the expression for the optimal cost follows.
\end{proof}

\begin{remark}
Proposition \ref{prop:optimalityfortranslatednorminvariantcontrol} shows that the Kantorovich OMT formulation associated with the GOMT problem \eqref{GOMTnorminvariant}, i.e., the GOMT problem for a norm-invariant system, has the same (scaled) squared-Euclidean ground cost as $c_{\mathrm{classical}}$ in \eqref{cclassical}. In particular,
$$c_{\mathrm{classical}}(\bm{x}_{0},\bm{x}_{\mathrm{f}}) = c_{\mathrm{norm-inv}}(\bm{x}_{0},\bm{x}_{\mathrm{f}}) = \dfrac{\|\bm{x}_{0} - \bm{x}_{\mathrm{f}}\|_2^2}{2\:t_{\mathrm{f}}}.$$
However, unlike the classical OMT, the geodesics for \eqref{OCPforNormInvariantGroundCostz} are not straight lines; see Fig. \ref{fig:feasible}. 
\end{remark}

Our results in Proposition 3 is atypical in the sense it is rare to have a closed-form expression for the ground cost $c$ induced by the associated deterministic optimal control problem. In comparison, existing results in the literature include: 
\begin{itemize}

\item the existence of solution when the cost does not admit {\it singular minimizing curves} or more weakly, {\it sharp extremals} \cite{agrachev2009optimal,figalli2010mass},

\item semi-closed form expression of linear quadratic costs for linear time invariant systems \cite{hindawi2011mass},

\item using change-of-variables to map GOMT over linear time-varying control systems to the Euclidean OMT with transformed endpoint PDF data \cite{chen2016optimal},

\item the existence-uniqueness of solutions and conditions for optimality for GOMT problems over feedback linearizable systems \cite{caluya2019finite,caluya2020finite},

\item restricting the support of the initial and target measures to equilibrium sets and using results on GOMT for convex ground costs \cite{elamvazhuthi2025optimal},

\item using the Benamou-Brenier formulation \cite{benamou2000computational} to establish existence of solutions for general nonlinear control systems \cite{elamvazhuthi2023dynamical,elamvazhuthi2024benamou}.
\end{itemize}

\section{Concluding Remarks}\label{sec:Conclusions}
In this work, we studied the ground cost for the optimal transport of angular velocity from a given initial to a given terminal joint PDF over a fixed deadline, and controlled angular velocity dynamics governed by the Euler's equation. The problem is of practical relevance, e.g., in designing guidance laws for stochastic stabilization of spin state for a spacecraft in the weak distributional sense. While existence-uniqueness and numerical simulation results for this generalized optimal mass transport problem has appeared in prior work \cite{yan2023optimal}, the associated Kantorovich optimal coupling formulation has not been studied before. This work fills this gap by deriving results for the ground cost for the Kantorovich formulation. 

The ground cost itself is the optimal value of certain structured deterministic nonlinear optimal control problem. We highlight how an analysis technique due to Athans et al. \cite{athans1963time} that combines different variants of the Cauchy-Schwarz inequality is relevant for this purpose. This is particularly noteworthy since analyzing Pontryagin's minimum principle in this setting becomes unwieldy. We establish that the analysis technique applies beyond the Eulerian angular velocity dynamics, and exactly determines the ground cost for a class of generalized optimal mass transport problems involving nonlinear dynamics with translated norm-invariant drift. The technique is very much in the spirit of the celebrated work \cite{steele2004cauchy}, and could be of broader interest.  


\section*{Acknowledgment}
We thank the reviewers' feedback which helped us fix several typos and significantly improved the presentation.

\balance

\bibliographystyle{IEEEtran}
\bibliography{References.bib}

\begin{thebibliography}{10}
\providecommand{\url}[1]{#1}
\csname url@samestyle\endcsname
\providecommand{\newblock}{\relax}
\providecommand{\bibinfo}[2]{#2}
\providecommand{\BIBentrySTDinterwordspacing}{\spaceskip=0pt\relax}
\providecommand{\BIBentryALTinterwordstretchfactor}{4}
\providecommand{\BIBentryALTinterwordspacing}{\spaceskip=\fontdimen2\font plus
\BIBentryALTinterwordstretchfactor\fontdimen3\font minus \fontdimen4\font\relax}
\providecommand{\BIBforeignlanguage}[2]{{%
\expandafter\ifx\csname l@#1\endcsname\relax
\typeout{** WARNING: IEEEtran.bst: No hyphenation pattern has been}%
\typeout{** loaded for the language `#1'. Using the pattern for}%
\typeout{** the default language instead.}%
\else
\language=\csname l@#1\endcsname
\fi
#2}}
\providecommand{\BIBdecl}{\relax}
\BIBdecl

\bibitem{villani2021topics}
C.~Villani, \emph{Topics in optimal transportation}.\hskip 1em plus 0.5em minus 0.4em\relax American Mathematical Soc., 2021, vol.~58.

\bibitem{villani2008optimal}
C.~Villani \emph{et~al.}, \emph{Optimal transport: old and new}.\hskip 1em plus 0.5em minus 0.4em\relax Springer, 2008, vol. 338.

\bibitem{benamou2000computational}
J.-D. Benamou and Y.~Brenier, ``A computational fluid mechanics solution to the {Monge-Kantorovich} mass transfer problem,'' \emph{Numerische Mathematik}, vol.~84, no.~3, pp. 375--393, 2000.

\bibitem{yan2023optimal}
C.~Yan, I.~Nodozi, and A.~Halder, ``Optimal mass transport over the {E}uler equation,'' in \emph{2023 62nd IEEE Conference on Decision and Control (CDC)}.\hskip 1em plus 0.5em minus 0.4em\relax IEEE, 2023, pp. 6819--6826.

\bibitem{figalli2007optimal}
A.~Figalli, ``Optimal transportation and action-minimizing measures,'' Ph.D. dissertation, Lyon, École normale supérieure (sciences), 2007.

\bibitem{kantorovich1942translocation}
L.~V. Kantorovich, ``On the translocation of masses,'' in \emph{Dokl. Akad. Nauk. USSR (NS)}, vol.~37, 1942, pp. 199--201.

\bibitem{elamvazhuthi2024benamou}
K.~Elamvazhuthi, ``Benamou-brenier formulation of optimal transport for nonlinear control systems on rd,'' \emph{arXiv preprint arXiv:2407.16088}, 2024.

\bibitem{athans1963time}
M.~Athans, P.~Falb, and R.~Lacoss, ``Time-, fuel-, and energy-optimal control of nonlinear norm-invariant systems,'' \emph{IEEE transactions on automatic control}, vol.~8, no.~3, pp. 196--202, 1963.

\bibitem{agrachev2009optimal}
A.~Agrachev and P.~Lee, ``Optimal transportation under nonholonomic constraints,'' \emph{Transactions of the American Mathematical Society}, vol. 361, no.~11, pp. 6019--6047, 2009.

\bibitem{figalli2010mass}
A.~Figalli and L.~Rifford, ``Mass transportation on sub-{R}iemannian manifolds,'' \emph{Geometric and functional analysis}, vol.~20, pp. 124--159, 2010.

\bibitem{hindawi2011mass}
A.~Hindawi, J.-B. Pomet, and L.~Rifford, ``Mass transportation with {LQ} cost functions,'' \emph{Acta applicandae mathematicae}, vol. 113, pp. 215--229, 2011.

\bibitem{chen2016optimal}
Y.~Chen, T.~T. Georgiou, and M.~Pavon, ``Optimal transport over a linear dynamical system,'' \emph{IEEE Transactions on Automatic Control}, vol.~62, no.~5, pp. 2137--2152, 2016.

\bibitem{caluya2019finite}
K.~F. Caluya and A.~Halder, ``Finite horizon density control for static state feedback linearizable systems,'' \emph{arXiv preprint arXiv:1904.02272}, 2019.

\bibitem{caluya2020finite}
K.~F. Caluya and A.~Halder, ``Finite horizon density steering for multi-input state feedback linearizable systems,'' in \emph{2020 American Control Conference (ACC)}.\hskip 1em plus 0.5em minus 0.4em\relax IEEE, 2020, pp. 3577--3582.

\bibitem{elamvazhuthi2025optimal}
K.~Elamvazhuthi and M.~Jacobs, ``Optimal transport of linear systems over equilibrium measures,'' \emph{Automatica}, vol. 175, p. 112222, 2025.

\bibitem{elamvazhuthi2023dynamical}
K.~Elamvazhuthi, S.~Liu, W.~Li, and S.~Osher, ``Dynamical optimal transport of nonlinear control-affine systems,'' \emph{Journal of Computational Dynamics}, vol.~10, no.~4, pp. 425--449, 2023.

\bibitem{steele2004cauchy}
J.~M. Steele, \emph{The {C}auchy-{S}chwarz master class: an introduction to the art of mathematical inequalities}.\hskip 1em plus 0.5em minus 0.4em\relax Cambridge University Press, 2004.

\end{thebibliography}

\end{document}